\documentclass[final,1p,number,sort&compress]{elsarticle}

\usepackage{amssymb}
\usepackage{amsmath}
\usepackage{amsthm}
\usepackage{mathabx}
\usepackage{accents}
\usepackage{booktabs}

\usepackage{color}

\newcommand{\diag}{\operatorname{diag}}

\newcommand{\assgn}{\ensuremath\mathrel{\mathop:}=}
\newcommand{\hpm}{\hphantom{-}}
\newcommand{\hpz}{\hphantom{0}}
\def\widebreve#1{\smash{\hbox{\Large$\skew2\breve{\hbox{\normalsize$#1$}}$}}\vphantom{#1}}
\numberwithin{equation}{section}
\journal{Linear Algebra Appl.}
\newtheorem{theorem}{Theorem}[section]

\newtheorem{lemma}[theorem]{Lemma}
\newtheorem{proposition}[theorem]{Proposition}
\newtheorem{corollary}[theorem]{Corollary}
\newtheorem{remark}[theorem]{Remark}
\newtheorem{example}[theorem]{Example}

\begin{document}
\begin{frontmatter}

\title{Nearly optimal scaling in the SR decomposition\tnoteref{label1}}
\tnotetext[label1]{Sanja Singer has been fully supported by Croatian
  Science Foundation under the project IP-2014-09-3670.  Miroslav
  Rozlo\v{z}n\'{\i}k has been supported by the Czech Science
  Foundation grant 20-01074S in the framework of RVO 67985840.}

\author[label2]{Heike Fa\ss bender}
\ead{h.fassbender@tu-braunschweig.de}
\address[label2]{Institute for Numerical Analysis,
  Technische Universit\"{a}t Braunschweig, Universit\"{a}tsplatz 2,
  38106 Braunschweig, Germany}

\author[label3]{Miroslav Rozlo\v{z}n\'{\i}k\corref{cor1}}
\cortext[cor1]{Corresponding author.}
\ead{miro@math.cas.cz}
\address[label3]{Institute of Mathematics, Czech Academy of Sciences,
  \v{Z}itn\'{a} 25, CZ-115 67 Prague 1, Czech Republic}

\author[label4]{Sanja Singer}
\ead{ssinger@fsb.hr}
\address[label4]{University of Zagreb,
  Faculty of Mechanical Engineering and Naval Architecture,
  Ivana Lu\v{c}i\'{c}a 5, 10000 Zagreb, Croatia}

\begin{abstract}
  In this paper we analyze the nearly optimal block diagonal scalings of the
  rows of one factor and the columns of the other factor in the
  triangular form of the SR decomposition.  The result is a block
  generalization of the result of the van der Sluis about the almost
  optimal diagonal scalings of the general rectangular matrices.
\end{abstract}


\end{frontmatter}

%
%
\section{Introduction}
\label{sec:1}
%
%

The QR factorization and the closely related QR algorithm are one of
the workhorses in solving general eigenvalue problems.  It is
well-known that the QR algorithm preserves the symmetric structure of
the matrix whose eigenvalues are to be computed such that the computed
eigenvalues will all be real (even so rounding errors are
unavoidable).  Unfortunately, there are a number of structured problems
whose structure is not preserved by the QR algorithm.  Thus, general
QR-like methods, in which the QR factorizations are replaced by other
factorizations have been studied by several authors, see, e.g.,
\cite{Watkins-Elsner-91}.  Here we consider the SR decomposition
which can be used in the SR algorithm which preserves the symplectic
as well as the Hamiltonian structure.

For a matrix $G \in \mathbb{R}^{2m, 2m}$ an SR decomposition
is given by
\begin{equation}
  G = \widetilde{S} \widetilde{R}
  = \widetilde{S} \begin{pmatrix}
    \widetilde{R}_{11} & \widetilde{R}_{12} \\
    \widetilde{R}_{21} & \widetilde{R}_{22}
  \end{pmatrix},
  \label{1.1}
\end{equation}
where $\widetilde{S}$ is symplectic, i.e.,
$\widetilde{S}^T J \widetilde{S} = J$
for the skew-symmetric matrix $J$ defined as
\begin{displaymath}
  J = \begin{pmatrix}
    \hpm 0 & I \\
    - I & 0
  \end{pmatrix}
  \in \mathbb{R}^{2m, 2m}.
\end{displaymath}
As usual, $I \in \mathbb{R}^{m, m}$ denotes the identity matrix.
The matrix $\widetilde{R}$ is $J$-triangular, that is, $R_{ij}$ are
upper triangular, and $R_{21}$ has zero diagonal.  The SR
decomposition \eqref{1.1} exists if all leading submatrices of even
dimension of $P G^T J G P^T$ are nonsingular (see, e.g.,
\cite[Theorem 11]{Elsner-79} or~\cite[Theorem 3.8]{BunseGerstner-86}),
and $P$ is the (perfect shuffle) permutation matrix
\begin{displaymath}
  P = (e_1, e_3, \ldots, e_{2m-1}, e_2, e_4, \ldots, e_{2m}),
\end{displaymath}
where $e_k$, $k = 1, \ldots, m$ are vectors of the canonical basis.
The set of $2m \times 2m$ SR decomposable matrices is thus dense in
$\mathbb{R}^{2m, 2m}$.

The SR decomposition is not unique as with $G = \widetilde{S} \widetilde{R}$
also $G = \undertilde{S} \undertilde{R}$ is an SR decomposition of $G$
where $\undertilde{S} = \widetilde{S} \widetilde{D}^{-1}$ and
$\undertilde{R} = \widetilde{D} \widetilde{R}$ for a matrix
\begin{equation}
  \widetilde{D} = \begin{pmatrix}
    C & F\hphantom{{}^{-1}} \\
    0 & C^{-1}
  \end{pmatrix},
  \label{1.2}
\end{equation}
with diagonal matrices $C, F \in \mathbb{R}^{m, m}$.  If uniqueness is
required, there are various possibilities how to make it unique by
adding requirements on $S$ or $\widetilde{R}$ (see, e.g.,
\cite{Fassbender-Razloznik-2016} for a summary of the typical
suggestions).

Symplectic matrices may be arbitrarily ill-conditioned.  Thus, one is
interested in making use of the non-uniqueness of the SR decomposition
by choosing $\widetilde{S}$ (or $\widetilde{R}$) factor so that its
condition is as good as possible.  Some first-order componentwise and
normwise perturbation bounds for a certain unique SR decomposition
($\diag (R_{11}) = |\diag (R_{22})|$, $\diag (R_{21}) = 0$) can be
found in~\cite{Chang-98} (see also~\cite{Demmel-Kagstroem-87}, while
in~\cite{Fassbender-Razloznik-2016} it is discussed how to choose the
entries of the $2 \times 2$ submatrices
\begin{displaymath}
  \begin{pmatrix}
    (\widetilde{R}_{11})_{jj} & (\widetilde{R}_{12})_{jj} \\
    0 & (\widetilde{R}_{22})_{jj}
  \end{pmatrix}
\end{displaymath}
of the $J$-triangular matrix $\widetilde{R}$ in order to  minimize the
condition number of $\widetilde{R}$ or the condition number of
$\widetilde{S}$.

Assume that $G = \widetilde{S} \widetilde{R}$ is a SR decomposition of
$G$.  We will consider the question on how to choose the matrix
$\widetilde{D}$ as in \eqref{1.2} such that the SR decomposition
\begin{displaymath}
  G = \undertilde{S} \undertilde{R},\quad
  \undertilde{S} = \widetilde{S} \widetilde{D}^{-1}, \quad
  \undertilde{R} = \widetilde{D} \widetilde{R}
\end{displaymath}
of $G$ has either an nearly optimally conditioned $\undertilde{S}$ or
an nearly optimally conditioned $\undertilde{R}$.  In particular, we
try to answer the questions on how to choose $\widetilde{D}_r$ and
$\widetilde{D}_c$ such that
\begin{equation}
  \kappa_2 (\widetilde{D}_r \widetilde{R})
  \leq \alpha_R \min_{\widetilde{D} \in \widetilde{\mathcal{D}}} (\widetilde{D} \widetilde{R})
  \label{1.3}
\end{equation}
and
\begin{equation}
  \kappa_2 (\widetilde{S} (\widetilde{D}_c)^{-1})
  \leq \alpha_C \min_{\widetilde{D} \in \widetilde{\mathcal{D}}} (\widetilde{S} \widetilde{D}^{-1}),
  \label{1.4}
\end{equation}
where $\widetilde{\mathcal{D}}$ denotes the set of all nonsingular
$2m \times 2m$ matrices of the form \eqref{1.2}, and
$\alpha_R, \alpha_C \in \mathbb{R}$.

It is well-known that equilibration tends to reduce the condition
number of a matrix.  Equilibration means the scaling of the rows
(and/or columns) of a matrix such that the norms of all rows (and/or
columns) obtain equal norms.  This has already been studied by van der
Sluis in~\cite{vanderSluis-69} (see also~\cite{HighamN-2002}). If
$G \in \mathbb{R}^{m, n}$ is a full rank matrix, than
\begin{gather*}
  \kappa_2 (\Sigma_r G) \leq \sqrt{m} \min_{\Sigma \in \mathcal{S}_c} (\Sigma G) \\
\noalign{\noindent for}
  \Sigma_r^{} = \diag(\| G e_1^{} \|_2^{-1}, \ldots, \| G e_n^{} \|_2^{-1})
\end{gather*}
and
\begin{gather*}
  \kappa_2 (G \Sigma_c^{}) \leq \sqrt{n} \min_{\Sigma \in \mathcal{S}_r} (G \Sigma) \\
\noalign{\noindent{for}}
  \Sigma_c^{} = \diag(\| e_1^T G \|_2^{-1}, \ldots, \| e_m^T G \|_2^{-1}),
\end{gather*}
where $\mathcal{S}_k$ denotes the set of all nonsingular $k \times k$
diagonal matrices and $e_k$ the $k$th column of the identity matrix.
In this paper we will generalize these results.

To be precise, we will consider not just the scaling of the SR
decomposition of square matrices $G \in \mathbb{R}^{2m, 2m}$, but we
will allow for rectangular $G \in \mathbb{R}^{2m,2n}$ where $m \geq n$.
Its standard SR decomposition is given by
\begin{displaymath}
  G = \widetilde{S} \widetilde{R}
  = \widetilde{S} \left( \begin{array}{c|c}
    \widetilde{R}_{11} & \widetilde{R}_{12} \\
    0_{m-n} & 0_{m-n} \\
    \hline
    \widetilde{R}_{21} & \widetilde{R}_{22} \\
    0_{m-n} & 0_{m-n}
  \end{array} \right)
\end{displaymath}
where $\widetilde{S} \in \mathbb{R}^{2m, 2m}$ is symplectic,
$\widetilde{R}_{11}, \widetilde{R}_{12}, \widetilde{R}_{22} \in \mathbb{R}^{n, n}$
are upper triangular, $\widetilde{R}_{21} \in \mathbb{R}^{n, n}$ is
upper triangular with zero diagonal and $0_{m-n} \in \mathbb{R}^{m - n, n}$
denotes a zero matrix.

The rest of the paper is organized as follows.  In Section~\ref{sec:2}
some preliminary observations are given which will be helpful for the
later discussion.  In Section~\ref{sec:3} we find the almost optimal
block-diagonal scaling from the left-hand side of the triangular
factor $R$ in the SR decomposition.  Section~\ref{sec:4} contains
similar results for the right-hand block-diagonal scalings of the
symplectic factor $S$.  In Section \ref{sec:5} some connections to
other types of factorizations are given.  In particular, the
symplectic QR factorization~\cite{SingerSanja-SingerSasa-2003} and the
Cholesky-like factorization of skew-symmetric matrices presented
in~\cite{Bunch-82} (see
also~\cite{Benner-Byers-Fassbender-Mehrmann-Watkins-2000}) are
considered. The results obtained in Sections \ref{sec:3} and
\ref{sec:4} apply immediately.  In the final section the theoretical
results are illustrated on four examples -- two for column scalings of
the triangular factor $R$ and two for the scalings of the factor
permuted symplectic factor $S$, respectively.

%
%
\section{Preliminary lemmata}
\label{sec:2}
%
%

Before we tackle these two problems in the next sections, we will
derive two helpful lemmata.  The first lemma is a straightforward
consequence of the Leibniz formula for the determinant of a $2 \times 2$
matrix.
\begin{lemma}
  For all matrices $B \assgn (B_1, B_2)$, $B_1, B_2 \in \mathbb{R}^{m}$
  it holds
  \begin{displaymath}
    \det(B_{}^T B) = \| B_1^{} \|_2^2 \| B_2^{} \|_2^2 - (B_1^T B_2^{})^2.
  \end{displaymath}
  \label{lm2.1}
\end{lemma}

Next we will proof a formulae for the condition number of a $2 \times 2$
matrix.  For this, we make use of the following well-known facts
(see, e.g., \cite{Golub-VanLoan-96}) for $A, B \in \mathbb{R}^{n, n}$
and the singular value decomposition
$B = U \Sigma V^T$ with $U^T U = V^T V = I$,
$\Sigma = \diag(\sigma_1(B), \ldots, \sigma_n(B))$:
\begin{gather*}
  \det(A B) = \det(A) \det(B), \qquad
  \det(B^T) = \det(B), \\
  \det(B) = \prod_{k = 1}^n \sigma_k(B), \qquad
  \|B\|_F = \sum_{k = 1}^n \sigma_k^2(B), \qquad
  \|B\|_2 = \sigma_{\max}(B).
\end{gather*}

\begin{lemma}
  For any matrix $B \in \mathbb{R}^{2, 2}$ its spectral condition
  number in terms of its determinant and and Frobenius norm can be
  written as
  \begin{displaymath}
    \kappa_2 (B) = \frac{\sigma_{\max}(B)}{\sigma_{\min}(B)}
    = \frac{ \| B \|_F^2 + \sqrt{\| B \|_F^4 - 4 \det^2 (B)}}{2 |\det(B)|}
  \end{displaymath}
  where $\sigma_{\max}(B)$ and $\sigma_{\min}(B)$ are the maximal and
  minimal singular values of $B$.
  \label{lm2.2}
\end{lemma}

\begin{proof}
  For $B \in \mathbb{R}^{2, 2}$ we have
  \begin{equation}
    \| B \|_F^2 = \sigma_{\max}^2(B) + \sigma_{\min}^2(B)
    \label{2.1}
  \end{equation}
  and
  \begin{equation}
    \det(B^T B) = {\det}^2(B) = \sigma_{\max}^2(B) \cdot \sigma_{\min}^2(B).
    \label{2.2}
  \end{equation}
  Note that \eqref{2.1} and \eqref{2.2} are Vieta's formulas for
  the sum and the product of the roots $\sigma_{\max}^2(B)$ and
  $\sigma_{\min}^2(B)$ of the quadratic equation
  \begin{displaymath}
    \big( \tau - \sigma_{\max}^2(B) \big) \big(\tau - \sigma_{\min}^2(B)\big)
    = \tau^2 - \| B \|_F^2 \tau + {\det}^2(B) = 0.
  \end{displaymath}
  Therefore, squares of the singular values can be written by using the
  coefficients of the polynomial,
  \begin{align*}
    \sigma_{\max}^2 (B) & = \frac{\| B \|_F^2 + \sqrt{\| B \|_F^4 - 4 \det^2(B)}}{2}, \\
    \sigma_{\min}^2 (B) & = \frac{\| B \|_F^2 - \sqrt{\| B \|_F^4 - 4 \det^2(B)}}{2}.
  \end{align*}
 Hence, the spectral condition number of $B$ can be expressed as
  \begin{displaymath}
    \kappa_2 (B) = \frac{\sigma_{\max}(B)}{\sigma_{\min}(B)}
    = \frac{\sigma^2_{\max}(B)}{|\det(B)|}
    = \frac{\| B \|_F^2 + \sqrt{\| B \|_F^4 - 4 \det^2(B)}}{2|\det(B)|}.
    \qedhere
  \end{displaymath}
\end{proof}

%
%
\section{Nearly optimal block-row scaling of $\widetilde{R}$}
\label{sec:3}
%
%

Now we are ready to consider the problem \eqref{1.3}.  It is easy to
see that for a $J$-triangular matrix $\widetilde{R} \in \mathbb{R}^{2n, 2n}$
the permuted matrix $\widetilde{P} \widetilde{R} \widetilde{P}^T$ is
an upper triangular matrix.  Similarly, a matrix
$\widetilde{D} \in \mathbb{R}^{2n, 2n}$ of the form \eqref{1.2} is
permuted to the block diagonal matrix
\begin{equation}
  D = \widetilde{P} \widetilde{D} \widetilde{P}^T
  = \diag \left(
  \begin{pmatrix}
    c_{11}^{} & f_{11}^{} \\
    0       & c_{11}^{-1}
  \end{pmatrix}, \ldots,
  \begin{pmatrix}
    c_{nn}^{} & f_{nn}^{} \\
    0       & c_{nn}^{-1}
  \end{pmatrix} \right) \in \mathbb{R}^{2n, 2n}.
  \label{3.1}
\end{equation}
As
\begin{displaymath}
  DR = (\widetilde{P} \widetilde{D} \widetilde{P}^T)
  (\widetilde{P} \widetilde{R} \widetilde{P}^T)
  = \widetilde{P} \widetilde{D} \widetilde{R} \widetilde{P}^T
\end{displaymath}
and as the spectral norm is unitary invariant, we have
$\kappa_2(\widetilde{D} \widetilde{R}) = \kappa_2(DR)$.  Thus, instead
of \eqref{1.3} we will actually consider the following equivalent
problem.  Given an upper triangular matrix $R \in \mathbb{R}^{2n, 2n}$
find a matrix $D_r$ such that
\begin{equation}
  \kappa_2(D_r R) \leq \alpha_R \min_{D \in \mathcal{D}} (D R)
  \label{3.2}
\end{equation}
where $\mathcal{D}$ denotes the set of all nonsingular $2n \times 2n$
matrices of the form \eqref{3.1} and $\alpha_R \in \mathbb{R}$.

As any $D \in \mathcal{D}$ is a block diagonal matrix with $2 \times 2$
blocks on the diagonal, we will block $R$ accordingly
\begin{equation}
  R = \begin{pmatrix}
    R_{11} & \cdots & R_{1n} \\
    0 & \ddots & \vdots \\
    0 & 0 & R_{nn}
  \end{pmatrix},
\label{3.3}
\end{equation}
with $R_{ij} \in \mathbb{R}^{2, 2}$ for $i = 1, \ldots, j$,
$j = 1, \ldots, n$ and diagonal blocks
\begin{displaymath}
  R_{jj} = \begin{pmatrix}
    r_{11}^{(j)} & r_{12}^{(j)} \\
    0          & r_{22}^{(j)}
  \end{pmatrix}, \qquad r_{11}^{(j)} r_{22}^{(j)} \neq 0
\end{displaymath}
for $j = 1, \ldots, n$.  Thus, we will consider
\begin{displaymath}
  X = D_r R = \diag(D_1, \ldots, D_n) R,
\end{displaymath}
where $j$th block-row of the matrix $X$ is
\begin{equation}
  X_j= D_j \cdot \begin{pmatrix}
    0_2 & \cdots & 0_2 & R_{jj} & \cdots & R_{jn}
    \end{pmatrix} \in \mathbb{R}^{2, 2n},
  \label{3.4}
\end{equation}
and
\begin{displaymath}
  D_j = \begin{pmatrix}
    c_{jj}^{} & f_{jj}^{} \\
    0       & c_{jj}^{-1}
  \end{pmatrix},
\end{displaymath}
for $j = 1, \ldots, n$.

Let $L$ be
\begin{equation}
  L = R^T = (L_1, \ldots, L_n), \qquad L_j \in \mathbb{R}^{2n, 2}
  \label{3.5}
\end{equation}
such that $L_j^T$ denotes the $j$th block row of the matrix $R$.
Denote the two columns of $L_j$ by $L_{j1}$ and $L_{j2}$,
respectively,
\begin{displaymath}
  L_j = (L_{j1}, L_{j2}), \qquad L_{j1}, L_{j2} \in \mathbb{R}^{2n}.
\end{displaymath}

We will tackle our problem in three steps.  First we will see that it
is possible to choose $D_j$ such that $D_j$ minimizes the Frobenius
norm of $X_j$ and the two rows of $X_j$ have the same Frobenius norm
$\beta_j$.  Next we will discuss how to choose $D_r$ such that all row
of $X$ have the same Frobenius norm $\beta \geq \beta_j$.  Finally, we
will give an answer for \eqref{3.2}.

Thus, we start our discussion by first seeing what can be achieved
locally by looking at the $j$th block row of $X$.  We are looking for
$D_j$ that minimizes the Frobenius norm of $X_j$.

The Frobenius norm of $X_j$ can now be expressed as
\begin{align}
 \| X_j^{} \|_F^2 & = \| D_j^{}  L_j^T \|_F^2 = \| L_j^{} D_j^T \|_F^2
  = \left \| \begin{pmatrix}
    c_{jj}^{} L_{j1}^{} + f_{jj}^{} L_{j2}^{}, & c_{jj}^{-1} L_{j2}^{}
    \end{pmatrix} \right\|_F^2 \nonumber \\
  & = \| c_{jj}^{} L_{j1}^{} + f_{jj}^{} L_{j2}^{} \|_2^2
  + \| c_{jj}^{-1} L_{j2}^{} \|_2^2
    \label{3.6} \\
  & = c_{jj}^2 \| L_{j1}^{} \|_2^2 + 2 c_{jj}^{} f_{jj}^{} L_{j1}^T L_{j2}^{}
    + f_{jj}^2 \| L_{j2}^{} \|_2^2 + \frac{\| L_{j2}^{} \|_2^2 }{c_{jj}^2}.
  \label{3.7}
\end{align}
With this we are ready to state an optimal scaling $D_j$ for the
$j$th block row of $R$.
\begin{theorem}
  Let $R \in \mathbb{R}^{2n, 2n}$ as in \eqref{3.3} be given.  Let
  $L = R^T$ as in \eqref{3.5} and $X_j^{} = D_j^{} L_j^T$ as in
  \eqref{3.4}.  The Frobenius norm of $X_j$,
  $\| X_j^{} \|_F^{} = \| X_j^T \|_F^{} = \| L_j^{} D_j^T \|_F^{}$ is
  minimized for
  \begin{equation}
    \widehat{D}_j = \begin{pmatrix}
      \hat{c}_{jj}^{} & \hat{f}_{jj}^{} \\
      0             & \hat{c}_{jj}^{-1}
    \end{pmatrix},
    \label{3.8}
  \end{equation}
  where
  \begin{align}
    \hat{c}_{jj}^{} & = \frac{\| L_{j2}^{} \|_2}{\sqrt[4]{\det(L_j^T L_j^{})}},
    \label{3.9} \\
    \hat{f}_{jj}^{} & = -\frac{ L_{j1}^T L_{j2}^{}}{\| L_{j2}^{} \|_2 \sqrt[4]{\det(L_j^T L_j^{})}}.
    \label{3.10}
  \end{align}
  Thus, for the Frobenius norm of the $j$th block row of $R$ for the
  optimal $\widehat{D}_j$ it holds
  \begin{displaymath}
    \| X_j^{} \|_F = \| L_j^{} \widehat{D}_j^T \|_F^{} = \sqrt{2} \beta_j^{}
  \end{displaymath}
  with
  \begin{equation}
    \beta_j^{} \assgn \sqrt[4]{\det(L_j^T L_j^{})}.
    \label{3.11}
  \end{equation}
  \label{tm3.1}
\end{theorem}

\begin{proof}
  The partial derivatives of $\| X_j \|_F$ with respect to $c_{jj}$
  and $f_{jj}$ need to be equal to zero.  Differentiating \eqref{3.7}
  gives
  \begin{align*}
    0 & = c_{jj}^{} \| L_{j1}^{} \|_2^2 +  f_{jj}^{} L_{j1}^T L_{j2}^{}
    - \frac{\| L_{j2}^{} \|_2^2}{c_{jj}^3}, \\
    0 & = c_{jj}^{}  L_{j1}^T L_{j2}^{} + f_{jj}^{} \| L_{j2}^{} \|_2^2 .
  \end{align*}
  Rewriting the second equation as
  \begin{displaymath}
    f_{jj}^{} = -\frac{c_{jj}^{} L_{j1}^T L_{j2}^{}}{\| L_{j2}^{} \|_2^2}
  \end{displaymath}
  and substituting this expression into the first equation yields
  \begin{displaymath}
    0 = c_{jj}^4 \left( \| L_{j1}^{} \|_2^2
    - \frac{(L_{j1}^T L_{j2}^{})^2}{\| L_{j2}^{} \|_2^2} \right)
    - \| L_{j2}^{} \|_2^2,
  \end{displaymath}
  that is,
  \begin{displaymath}
    c_{jj}^4 = \frac{\| L_{j2}^{} \|_2^4}{(\| L_{j1}^{} \|_2^2 \| L_{j2}^{} \|_2^2
      - L_{j1}^T L_{j2}^{})_{}^2}.
  \end{displaymath}
  With Lemma \ref{lm2.1} we obtain \eqref{3.9}, and therefore
  \eqref{3.10}.  As the Hessian matrix
  \begin{displaymath}
    \begin{pmatrix}
      \| L_{j1}^{} \|_2^2 + 4 \det (L_j^T L_j^{}) &  L_{j1}^T L_{j2}^{} \\
      L_{j1}^T L_{j2}^{} & \| L_{j2}^{} \|_2^2
    \end{pmatrix}
  \end{displaymath}
  is symmetric positive definite (its trace and its determinants are
  positive), $\hat{c}_{jj}$ and $\hat{f}_{jj}$ as in \eqref{3.9} and
  \eqref{3.10} give the global minimum of
  $\min_{c_{jj}, f_{jj}} \|L_j^{} D_j^T \|_F$.

  By substituting the optimal $\hat{c}_{jj}$ and $\hat{f}_{jj}$ into
  \eqref{3.7} we obtain with Lemma \ref{lm2.1}
  \begin{align*}
    \| D_j^{} L_j^T \|_F^2
    & = \frac{\| L_{j2}^{} \|_2^2 \| L_{j1}^{} \|_2^2}{\sqrt{\det (L_j^T L_j^{})}}
      - 2 \frac{(L_{j1}^T L_{j2}^{})_{}^2}{\sqrt{\det (L_j^T L_j^{})}}
      + \frac{(L_{j1}^T L_{j2}^{})_{}^2}{\sqrt{\det (L_j^T L_j^{})}}
      + {\sqrt{\det (L_j^T L_j^{})}} \\
    & = \left( \frac{ \| L_{j2}^{} \|_2^2 \| L_{j1}^{} \|_2^2
      - (L_{j1}^T L_{j2}^{})_{}^2}{\sqrt{\det (L_j^T L_j^{})}}
      + {\sqrt{\det (L_j^T L_j^{})}} \right) \\
    & = \left( \frac{\det (L_j^T L_j^{})}{\sqrt{\det (L_j^T L_j^{})}}
      + {\sqrt{\det (L_j^T L_j^{})}} \right)
      = 2 \sqrt{\det (L_j^T L_j^{})}
      = 2 \beta_j^2.
    \qedhere
  \end{align*}
\end{proof}
It also holds that the two rows of $X_j$ have the same norm.
\begin{corollary}
  It holds that
  \begin{displaymath}
    \| e_1^T X_j^{} \|_2^{} = \| e_2^T X_j^{} \|_2^{} = \beta_j^{}.
  \end{displaymath}
  \label{cor3.2}
\end{corollary}

\begin{proof}
  Recall that
  \begin{displaymath}
    X_j^T = (c_{jj}^{} L_{j1}^{} + f_{jj}^{} L_{j2}^{}, c_{jj}^{-1} L_{j2}^{})
  \end{displaymath}
  holds.  By inserting value of $\hat{c}_{jj}$ from \eqref{3.9} into
  $\| \hat{c}_{jj}^{-1} L_{j2}^T \|_2^2$, it is easy to compute the
  squared norm of the second row of $X_j$,
  \begin{equation}
    \| \hat{c}_{jj}^{-1} L_{j2}^T \|_2^2 = \frac{\| L_{j2}^{} \|_2^2}{\hat{c}_{jj}^2}
    = \sqrt{\det (L_j^T L_j^{})}
    = \beta_j^2.
    \label{3.12}
  \end{equation}
  Therefore, from \eqref{3.6}, it follows that for the squared norm of
  the first row of $X_j$ that
  \begin{equation}
    \|c_{jj}^{} L_{j1}^{} + f_{jj}^{} L_{j2}^{} \|_2^2 = \beta_j^2,
    \label{3.13}
  \end{equation}
  holds, i.e., both rows of $X_j^{} = \widehat{D}_j^{} L_j^T$ have
  the same norm $\beta_j$.
\end{proof}

The spectral condition number of the matrix $\widehat{D}_j$ from
\eqref{3.8}, as well as the Frobenius condition number can be obtained
easily.
\begin{theorem}
  Let $\widehat{D}_j$ be as in Theorem~\ref{tm3.1}.  Then
  \begin{align*}
    \kappa_2^{} (\widehat{D}_j^{})
    & = \frac{\| L_j^{} \|_F^2 + \sqrt{\| L_j^{} \|_F^4 - 4 \det(L_j^T L_j^{})}}
      {2 \sqrt{\det (L_j^T L_j^{})}}, \\
    \kappa_F^{} (\widehat{D}_j^{})
    & = \frac{\| L_j^{} \|_F^2}{\sqrt{\det (L_j^T L_j^{})}}
      = \frac{\| L_j^{} \|_F^2}{\| L_j^{} \|_2 \sigma_{\min}(L_j^{})}.
  \end{align*}
  \label{tm3.3}
\end{theorem}

\begin{proof}
  The spectral condition number is a direct consequence of
  Lemma~\ref{lm2.2}
  \begin{displaymath}
    \kappa_2^{} (\widehat{D}_j^{})
    = \frac{\| \widehat{D}_j^{} \|_F^2
      + \sqrt{\| \widehat{D}_j^{} \|_F^4
      - 4 \det^2(\widehat{D}_j^{})}}{2 |\det(\widehat{D}_j^{})|}
  \end{displaymath}
  and the following observation obtained with the help of Lemma \ref{lm2.1}
  \begin{align*}
    \| \widehat{D}_j^{} \|_F^2
    & = c_{jj}^2 + e_{22}^2 + c_{jj}^{-2}
      = \frac{\| L_{j2}^{} \|_2^4 + (L_{j1}^T L_{j2}^{})_{}^2
      + \det (L_j^T L_j^{})}{\| L_{j2} ^{} \|_2^2 \sqrt{\det (L_j^T L_j^{})}} \\[6pt]
    & = \frac{\| L_{j2}^{} \|_2^4 + (L_{j1}^T L_{j2}^{})_{}^2
      + \| L_{j1}^{} \|_2^2 \|L_{j2}^{} \|_2^2
      - (L_j^T L_j^{})_{}^2}{\| L_{j2} \|_2^2 \sqrt{\det(L_j^T L_j^{})}} \\[6pt]
    & = \frac{\| L_{j2}^{} \|_2^2 + \| L_{j1}^{} \|_2^2}{\sqrt{\det (L_j^T L_j^{})}}
      = \frac{\| L_j^{} \|_F^2}{\sqrt{\det (L_j^T L_j^{}})}.
  \end{align*}
  The expression for
  \begin{displaymath}
    \kappa_F^{} (\widehat{D}_j^{})
    = \|\widehat{D}_j^{} \|_F^{} \| \widehat{D}_j^{-1} \|_F^{}
  \end{displaymath}
  follows immediately from \eqref{2.2} as
  $\| \widehat{D}_j^{-1} \|_F^{} = \| \widehat{D}_j^{} \|_F^{}$.
\end{proof}

\noindent
The following connection between columns $L_j$ and
the matrix $\widehat{D}_j^{-1}$ will be useful later~on.
\begin{proposition}
  Let $L_j$ be the $j$th block column of the matrix $R^T$ as in
  \eqref{3.5}, and $\widehat{D}_j$ as in Theorem~\ref{tm3.1}.
  Let the QL factorization (\cite{Golub-VanLoan-96}) of $L_j$ be given
  by
  \begin{displaymath}
    L_j = V_j \begin{pmatrix}
      0 \\
      \widehat{L}_{jj}
    \end{pmatrix}
  \end{displaymath}
  with the orthogonal matrix $V_j \in \mathbb{R}^{2n, 2n}$,
  $V_j^T V_j^{} = I_{2n}^{}$ and the lower triangular factor
  $\widehat{L}_{jj} \in \mathbb{R}^{2, 2}$,
  \begin{displaymath}
    \widehat{L}_{jj} = \begin{pmatrix}
      \hat{l}_{11}^{(j)} & 0 \\
      \hat{l}_{21}^{(j)} & \hat{l}_{22}^{(j)}
    \end{pmatrix}, \qquad \hat{l}_{11}^{(j)}, \hat{l}_{22}^{(j)} > 0.
  \end{displaymath}
  Then it holds for all $j = 1, \ldots, n$ that
  \begin{equation}
    \widehat{L}_{jj}^{} = \beta_j^{} \widehat{D}_j^{-T} .
    \label{3.14}
  \end{equation}
  \label{prop3.4}
\end{proposition}

\begin{proof}
  We immediately have
  \begin{displaymath}
    L_j^T L_j^{} = (0, \widehat{L}_{jj}^T)
    V_j^T V_j^{}
    \begin{pmatrix}
      0 \\
      \widehat{L}_{jj}^{}
    \end{pmatrix}
    = \widehat{L}_{jj}^T \widehat{L}_{jj}^{}.
  \end{displaymath}
  Then, from
  \begin{displaymath}
    L_j^T L_j^{} = \begin{pmatrix}
      \| L_{j1}^{} \|_2^2 & L_{j1}^T L_{j2}^{} \\
      L_{j1}^T L_{j2}^{} & \| L_{j2}^{} \|_2^2
    \end{pmatrix}
    = \widehat{L}_{jj}^T \widehat{L}_{jj}^{}
    = \begin{pmatrix}
      (\hat{l}_{11}^{(j)})_{}^2 + (\hat{l}_{21}^{(j)})_{}^2
         & \hat{l}_{21}^{(j)} \hat{l}_{22}^{(j)} \\
      \hat{l}_{21}^{(j)} \hat{l}_{22}^{(j)} & (\hat{l}_{22}^{(j)})_{}^2
    \end{pmatrix},
  \end{displaymath}
  it follows that
  \begin{gather*}
    \hat{l}_{22}^{(j)} = \| L_{j2}^{} \|_2, \qquad
    \hat{l}_{21}^{(j)} = \frac{L_{j1}^T L_{j2}^{}}{\| L_{j2}^{} \|_2}, \\
    \hat{l}_{11}^{(j)} = \frac{\sqrt{\| L_{j1}^{} \|_2^2 \| L_{j2}^{} \|_2^2
      - (L_{j1}^T L_{j2}^{})_{}^2}}{\| L_{j2}^{} \|_2^{}}
     = \frac{\sqrt{\det (L_j^T L_j^{})}}{ \| L_{j2}^{} \|_2^{}}.
  \end{gather*}
 With \eqref{3.8}--\eqref{3.11} we obtain
 \begin{displaymath}
   \hat{l}_{11}^{(j)} = \beta_j^{} \hat{c}_{jj}^{-1}, \qquad
   \hat{l}_{22}^{(j)} = \beta_j^{} \hat{c}_{jj}^{}, \qquad
   \hat{l}_{21}^{(j)} = -\beta_j^{} \hat{f}_{jj}^{},
 \end{displaymath}
 so that $\widehat{L}_{jj}^{} = \beta_j^{} \widehat{D}_j^{-T}$
 holds.
\end{proof}

The following lemma is an easy consequence of
Proposition~\ref{prop3.4}.  It will be helpful in proving the main
theorem of this section.
\begin{lemma}
  Let $L_j$ be the $j$th block column of the matrix $R^T$ defined by
  \eqref{3.5} with the QL factorization as in
  Proposition~\ref{prop3.4} and $\widehat{D}_j$ as in
  Theorem~\ref{tm3.1}.  For any matrix $B \in \mathbb{R}^{2, 2}$ it
  holds
  \begin{displaymath}
    \| B \widehat{D}_j^{-1} \|_2^{} = \frac{\| B L_j^T \|_2^{}}{\beta_j}.
  \end{displaymath}
  \label{lm3.5}
\end{lemma}

\begin{proof}
  From \eqref{3.14} it follows
  \begin{displaymath}
    B \widehat{D}_j^{-1} = \frac{1}{\beta_j} B \widehat{L}_{jj}^T,
  \end{displaymath}
  and by using the unitary invariance of the spectral norm we obtain
  \begin{displaymath}
    \| B \widehat{D}_j^{-1} \|_2^{}
    = \frac{\| B \widehat{L}_{jj}^T \|_2^{}}{\beta_j}
    = \frac{\left\| B (0, \widehat{L}_{jj}^T) \right\|_2^{}}{\beta_j}
    = \frac{\| B (0, \widehat{L}_{jj}^T) V_j^T \|_2^{}}{\beta_j}
    = \frac{\| B L_j^T \|_2^{}}{\beta_j}.
    \qedhere
  \end{displaymath}
\end{proof}

Our findings so far allow to construct a scaling matrix
$\widehat{D}_r = \diag (\widehat{D}_1, \ldots, \widehat{D}_n)$ such
that the Frobenius norm of each block row is minimized and the two
rows in the $j$th block row of $\widehat{D}_r R$ have the same
Frobenius norm $\beta_j$.  Our next goal is to determine a scaling
$\widetilde{D}_r = \diag (\widetilde{D}_1, \ldots, \widetilde{D}_n) \in \mathcal{D}$
such that (similarly to the result obtained by van der Sluis)
all rows of the matrix $\widetilde{D}_r R$ have the same Frobenius
norm equal to $\beta$.
\begin{theorem}
  Let $R \in \mathbb{R}^{2n, 2n}$ as in \eqref{3.3} be given.  Let
  $L = R^T$ be as in \eqref{3.5} and $D_j$, $j = 1, \ldots, n$ given as in
  \eqref{3.4}.  Let $\beta_j$ be as in Theorem~\ref{tm3.1}, and let
  $\beta \geq \beta_j$.  All rows of $\widetilde{D}_r R$ have the same
  norm $\beta$ for
  $\widetilde{D}_r = \diag (\widetilde{D}_1, \ldots, \widetilde{D}_n) \in \mathcal{D}$
  where
  \begin{equation}
    \widetilde{D}_j = \begin{pmatrix}
      \tilde{c}_{jj}^{} & \tilde{f}_{jj}^{} \\
      0 & \tilde{c}_{jj}^{-1}
    \end{pmatrix}
    \label{3.15}
  \end{equation}
  for $j = 1, \ldots, n$ with
  \begin{align}
    \tilde{c}_{jj}^{} & = \frac{\| L_{j2}^{} \|_2^{}}{\beta},
    \label{3.16} \\
    \tilde{f}_{jj}^{} & = \frac{-L_{j1}^T L_{j2}^{}
      \pm \sqrt{\beta_{}^4 - \beta_j^4}}{\beta \| L_{j2}^{} \|_2^{}}.
    \label{3.17}
  \end{align}
  \label{tm3.6}
\end{theorem}

\begin{proof}
  The requirement that all rows of $\widetilde{D} R = \widetilde{D} L^T$
  should have the same norm $\beta$ gives relations analogous to
  \eqref{3.12}--\eqref{3.13} for all $j = 1, \ldots, n$
  \begin{align}
    \beta_{}^2 & = \|\tilde{c}_{22}^{} L_{j2}^T\|_2^2
      = \frac{\| L_{j2}^{} \|_2^2}{\tilde{c}_{jj}^2}
      \label{3.18}\\
    \beta_{}^2 & = \| \tilde{c}_{jj}^{} L_{j1}^{}
        + \tilde{f}_{jj}^{} L_{j2}^{} \|_2^2
    = \tilde{c}_{jj}^2 \| L_{j1}^{} \|_2^2
      + 2 L_{j1}^T L_{j2}^{} \tilde{c}_{jj}^{} \tilde{f}_{jj}^{}
      + \tilde{f}_{jj}^2 \| L_{j2}^{} \|_2^2.
      \label{3.19}
  \end{align}
  Relation \eqref{3.18} immediately implies the choice of
  $\tilde{c}_{jj}$.

  Substituting \eqref{3.16} into \eqref{3.19} yields the quadratic
  equation for $\tilde{f}_{jj}$
  \begin{displaymath}
    \tilde{f}_{jj}^2
    + 2 \frac{L_{j1}^T L_{j2}^{}}{\beta \| L_{j2}^{} \|_2^{}} \tilde{f}_{jj}^{}
    + \frac{ \| L_{j1}^{} \|_2^2 }{\beta_{}^2}
    - \frac{\beta_{}^2}{\| L_{j2}^{} \|_2^2} = 0.
  \end{displaymath}
  If $\beta \geq \beta_j$, the equation has two real solutions
  \eqref{3.17},
  \begin{align*}
    \tilde{f}_{jj}^{} & = -\frac{L_{j1}^T L_{j2}^{}}{\beta \| L_{j2}^{} \|_2^{}}
      \pm \sqrt{\frac{(L_{j1}^T L_{j2}^{})_{}^2
      - \| L_{j1}^{} \|_2^2 \| L_{j2}^{} \|_2^2
      + \beta_{}^4}{\beta_{}^2 \| L_{j2}^{} \|_2^2}} \\
    & = -\frac{L_{j1}^T L_{j2}^{}}{\beta \| L_{j2}^{} \|_2}
      \pm \sqrt{\frac{-\det (L_{j1}^T L_{j2}^{})
      + \beta_{}^4}{\beta_{}^2 \| L_{j2}^{} \|_2^2}}
    = -\frac{L_{j1}^T L_{j2}^{}}{\beta \| L_{j2}^{} \|_2^{}}
      \pm \sqrt{\frac{-\beta_j^4 + \beta_{}^4}{\beta_{}^2 \| L_{j2}^{} \|_2^2}}
  \end{align*}
  with $\beta_j^2 = \sqrt{\det (L_{j1}^T L_{j2}^{})}$ as in \eqref{3.11}.
\end{proof}

It is not possible to achieve the a row scaling with a diagonal block scaling.
\begin{remark}
  If instead of the upper triangular $\widetilde{D}_j$ as in the previous
  theorem a diagonal block scaling matrix of the form
  \begin{displaymath}
    \widetilde{D}_j^{} = \diag(\tilde{c}_{jj}^{}, \tilde{c}_{jj}^{-1})
  \end{displaymath}
  is used, then it is not always possible to find $\tilde{c}_{jj}$
  such that the rows of the matrix $\widetilde{D} L^T$ have equal norms.
  \label{rem3.7}
\end{remark}

\begin{proof}
  The requirement that all rows of $\widetilde{D} R = \widetilde{D} L^T$
  should have the same norm $\beta$ gives, in analogy to
  \eqref{3.12}--\eqref{3.13} and \eqref{3.18}--\eqref{3.19} for all
  $j = 1, \ldots, n$
  \begin{displaymath}
    \tilde{c}_{jj} \| L_{j1} \|_2 = \beta, \qquad
    \frac{\| L_{j2} \|_2}{\tilde{c}_{jj}} = \beta.
  \end{displaymath}
  These two equations imply that the products $\| L_{j2} \|_2 \| L_{j1} \|_2$
  have to be identical for all indices $j$, which is only valid for
  very special cases.
\end{proof}

Now we are ready for the main theorem in the section.  Taking any
\begin{displaymath}
  \beta \geq \max_{j = 1, \ldots, n} \{ \beta_j \}
\end{displaymath}
Theorem~\ref{tm3.6} gives a block scaling $\widetilde{D}_r$ such that
all rows of the matrix $\widetilde{D}_r R$ have the same norm equal to
$\beta$.  Indeed, its condition number could be close to the optimal
scaling as it is in the standard case due to the result of van
der Sluis.
\begin{theorem}
  Let $R \in \mathbb{R}^{2n, 2n}$ as in \eqref{3.3} be given.  Let
  $L = R^T$ be as in \eqref{3.5} and $D_j$, $j = 1, \ldots, n$ given as
  \eqref{3.4}.  Let $\widehat{D}_j$, $j = 1, \ldots, n$ be as in
  \eqref{3.8} and Theorem~\ref{tm3.1}.  Let $\beta_j$, $j = 1, \ldots, n$
  be as in Theorem~\ref{tm3.1}.  Finally, let $\beta$ and $\gamma$ be
  defined as
  \begin{equation}
    \beta \assgn \max_{j = 1, \ldots, n} \{ \beta_j \}, \qquad
    \gamma \assgn \min_{j = 1, \ldots, n} \{ \beta_j \}.
    \label{3.20}
  \end{equation}
  Let $\widetilde{D}_r$ and $\widetilde{D}_j$, $j = 1, \ldots, n$
  be as in \eqref{3.15} and Theorem~\ref{tm3.6}.  Then
  $\widetilde{D}_r R$ is nearly optimally scaled.  More precisely, it
  holds
  \begin{displaymath}
    \min_{D \in \mathcal{D}} \kappa_2 (DR) \leq \kappa_2 (\widetilde{D}_r R)
    \leq \sqrt{2n} \,
    \frac{\beta \sqrt{\beta^2 + \sqrt{\beta^4 - \gamma^4}}}{\gamma^2}
    \min_{D \in \mathcal{D}} \kappa_2(D R).
  \end{displaymath}
  \label{tm3.8}
\end{theorem}

\begin{proof}
  According to Theorem~\ref{tm3.6} all rows of the matrix
  $X = \widetilde{D}_r R$ have the same norm $\beta$.  Therefore,
  \begin{equation}
    \| X \|_2 = \| \widetilde{D}_r R \|_2
    \leq \| \widetilde{D}_r R \|_F = \sqrt{2n} \, \beta.
    \label{3.21}
  \end{equation}

  In order to be able to give a bound on
  $\kappa_2^{}(X) = \| X \|_2^{} \| X_{}^{-1}\|_2^{}$ we need to find
  a bound on  $\|X_{}^{-1}\|_2^{}$.  Since the spectral norm is
  submultiplicative, for any nonsingular matrix $D$ we have
  \begin{equation}
    \| X_{}^{-1} \|_2^{} = \| R_{}^{-1} \widetilde{D}_r^{-1} \|_2^{}
    \leq \| R_{}^{-1} D_{}^{-1} \|_2^{}  \cdot \| D \widetilde{D}_r^{-1} \|_2^{}.
    \label{3.22}
  \end{equation}
  In particular, this holds for a block-diagonal matrix
  $D = \diag(D_1, \ldots, D_n) \in \mathcal{D}$.  With this, we have
  \begin{displaymath}
    D \widetilde{D}_r^{-1} = \diag(D_1 \widetilde{D}_1^{-1}, \ldots,
      D_n \widetilde{D}_n^{-1})
  \end{displaymath}
  and
  \begin{equation}
    \| D \widetilde{D}_r^{-1} \|_2^{}
    = \max_{j = 1, \ldots, n} \| D_j^{} \widetilde{D}_j^{-1} \|_2^{}
    \leq \max_{j = 1, \ldots, n} (\| D_j^{} \widehat{D}_j^{-1} \|_2^{}
      \| \widehat{D}_j^{} \widetilde{D}_j^{-1} \|_2^{})
    \label{3.23}
  \end{equation}
  for $\widehat{D}_j$, $j = 1, \ldots, n$ as in \eqref{3.8}.  From
  Lemma~\ref{lm3.5} with $B = D_j$  we obtain
  \begin{equation}
    \| D_j^{} \widehat{D}_j^{-1} \|_2^{}
    = \frac{\| D_j^{} L_j^T \|_2^{}}{\beta_j}.
    \label{3.24}
  \end{equation}
  Estimation of $\| \widehat{D}_j^{} \widetilde{D}_j^{-1}\|_2^{}$
  is more tedious.  A straightforward calculation shows that
  \begin{displaymath}
    \widehat{D}_j^{} \widetilde{D}_j^{-1}
    = \begin{pmatrix}
      \frac{\beta}{\beta_j}
      & \pm \frac{\sqrt{\beta_{}^4 - \beta_j^4}}{\beta \beta_j} \\[6pt]
      0 & \frac{\beta_j}{\beta}
    \end{pmatrix}.
  \end{displaymath}
  In order to determine $\| \widehat{D}_j^{} \widetilde{D}_{}^{-1} \|_2^2$
  we compute
  \begin{displaymath}
    (\widehat{D}_j^{} \widetilde{D}_j^{-1})_{}^T  \widehat{D}_j^{}
      \widetilde{D}_j^{-1}
    = \begin{pmatrix}
      \frac{\beta^2}{\beta_j^2}
        & \pm \frac{\sqrt{\beta_{}^4 - \beta_j^4}}{\beta_j^2} \\[9pt]
      \pm \frac{\sqrt{\beta_{}^4 - \beta_j^4}}{\beta_j^2}
        & \frac{\beta^2}{\beta_j^2}
    \end{pmatrix},
  \end{displaymath}
  its characteristic polynomial
  \begin{displaymath}
    0 = \left( \frac{\beta^2}{\beta_j^2}  - \lambda \right)^2
    - \frac{\beta_{}^4 - \beta_j^4}{\beta_j^4}
    = \lambda^2 - 2\frac{\beta^2}{\beta_j^2} \lambda + 1,
  \end{displaymath}
  and the roots
  \begin{displaymath}
    \lambda_{1, 2} = \frac{\beta^2}{\beta_j^2}
    \pm \sqrt{\frac{\beta_{}^4 - \beta_j^4}{\beta_j^4}}.
  \end{displaymath}
  Thus,
  \begin{equation}
    \| \widehat{D}_j^{} \widetilde{D}_j^{-1} \|_2^2
    = \frac{\beta_{}^2 + \sqrt{\beta_{}^4 - \beta_j^4}}{\beta_j^2}.
    \label{3.25}
  \end{equation}
  By inserting \eqref{3.24}--\eqref{3.25} into \eqref{3.23} we obtain
  \begin{align}
    \| D \widetilde{D}_r^{-1} \|_2^{}
    & \leq \max_{j = 1, \ldots, n} (\| D_j^{} \widehat{D}_j^{-1} \|_2^{}
    \| \widehat{D}_j^{} \widetilde{D}_j^{-1} \|_2^{})
    = \max_{j = 1, \ldots, n} \frac{\sqrt{\beta_{}^2
      + \sqrt{\beta_{}^4 - \beta_j^4}}}{\beta_j^2} \| D_j^{} L_j^T \|_2^{}
      \nonumber \\
    & \leq \frac{\sqrt{\beta^2 + \sqrt{\beta^4 - \gamma^4}}}{\gamma^2}
      \max_{j = 1, \ldots, n} \| D_j^{} L_j^T \|_2^{},
    \label{3.26}
  \end{align}
  with $\gamma$ as in \eqref{3.20}.

  As $D_j^{} L_j^T$ represent the $j$th block row of $DR$ we can write
  \begin{displaymath}
    D_j^{} L_j^T = M_j^{} DR,
  \end{displaymath}
  with
  \begin{displaymath}
    M_j = (e_{2j-1}, e_{2j})^T.
  \end{displaymath}
  Since the spectral norm is submultiplicative and $ \| M_j \|_2 = 1$,
  we have
  \begin{displaymath}
    \| D_j^{} L_j^T \|_2^{} = \| M_j^{} DR \|_2^{}
    \leq \| M_j^{} \|_2^{} \| DR \|_2^{}
    = \| DR \|_2^{}
  \end{displaymath}
  for all $j = 1, \ldots, n$.  By inserting this result in
  \eqref{3.26} it holds
  \begin{equation}
    \| D \widetilde{D}_r^{-1} \|_2^{}
    \leq \frac{\sqrt{\beta^2 + \sqrt{\beta^4 - \gamma^4}}}{\gamma^2}
      \| DR \|_2^{}.
    \label{3.27}
  \end{equation}
  From \eqref{3.21}--\eqref{3.22} and \eqref{3.27} we obtain
  \begin{displaymath}
    \kappa_2 (\widetilde{D}_r R)
    \leq \sqrt{2n} \beta \frac{\sqrt{\beta^2
        + \sqrt{\beta^4 - \gamma^4}}}{\gamma^2} \kappa_2(DR).
  \end{displaymath}
  Since the previous formula is valid for all block diagonal matrices
  $D\in \mathcal{D}$ the statement of the theorem follows.
\end{proof}

%
%
\section{Nearly optimal block-column scaling of $\widetilde{S}$}
\label{sec:4}
%
%

In this section we consider the problem \eqref{1.4}.

As in the previous section, we will consider an equivalent problem
stated using permuted version of the matrices under consideration.  In
particular, we will make use of the permuted version of the matrix
$\widetilde{D}$ as in \eqref{3.1}, and of the permuted version $S$ of the
symplectic matrix $\widetilde{S}$, where
\begin{gather*}
  \widetilde{S} = (s_1, s_2, \ldots, s_{2n-1}, s_{2n}) \in \mathbb{R}^{2m, 2n} \\
  S = \widetilde{S} \widetilde{P}^T
    = (s_1, s_{n+1}, s_2, s_{n+2}, \ldots, s_n, s_{2n}).
\end{gather*}
For
\begin{displaymath}
  \widehat{J} \assgn P J P^T = \diag(J_1, \ldots, J_1) \in
  \mathbb{R}^{2m, 2m}, \qquad
  J_1= \begin{pmatrix}
    \hpm 0 & 1 \\
       - 1 & 0
  \end{pmatrix}
  \in \mathbb{R}^{2, 2}.
\end{displaymath}
it holds
\begin{displaymath}
  S^T J S = (\widetilde{S} \widetilde{P}^T)^T J \widetilde{S} \widetilde{P}^T
  = \widetilde{P}^{} \widetilde{J} \widetilde{P}^T = \widehat{J}(1:2n,1:2n),
\end{displaymath}
where
\begin{displaymath}
  \widetilde{J} = \begin{pmatrix}
    \hpm 0 & I \\
    - I & 0
  \end{pmatrix}
  \in \mathbb{R}^{2n, 2n}.
\end{displaymath}
As $SD = (\widetilde{S} \widetilde{P}^T)(\widetilde{P} \widetilde{D}
\widetilde{P}^T) = \widetilde{S} \widetilde{D} \widetilde{P}^T$ and
as the spectral norm is unitary invariant, we have
$\kappa_2(\widetilde{S} \widetilde{D}) = \kappa_2(SD)$.

Thus, instead of \eqref{1.4} we will consider the following problem.
Given a permuted symplectic matrix $S \in \mathbb{R}^{2m, 2n}$ with
$S^T J S = \widehat{J}(1:2n,1:2n)$ find a matrix $D_c$ such that
\begin{equation}
  \kappa_2^{} (S D_c^{-1}) \leq \alpha_C \min_{D \in \mathcal{D}} (S D_{}^{-1})
  \label{4.1}
\end{equation}
where $\mathcal{D}$ denotes the set of all nonsingular $2n \times 2n$
matrices of the form \eqref{3.1} and $\alpha_C \in \mathbb{R}$.

\begin{remark}
  The optimal choice $\widetilde{D}_r$ from Theorem~\ref{tm3.6} is in
  general not optimal for \eqref{4.1}, that is
  $\kappa_2^{}(S \widetilde{D}_r^{-1})$ is not always less or equal to
  $\alpha_C \min_{D \in \mathcal{D}} (S D^{-1})$.  See Example
  \ref{ex6.3} for an illustration.
  \label{rm4.1}
\end{remark}

We will proceed in three steps as in the previous section to find an
answer to \eqref{4.1}.  In the first step we look for upper triangular
blocks
\begin{equation}
  D_j^{-1} = \begin{pmatrix}
    c_{jj}^{-1} & -f_{jj}^{}\\
    0 & \hpm c_{jj}^{}
  \end{pmatrix}
  \label{4.2}
\end{equation}
such that they minimize the Frobenius norm of the product
$S_j^{} D_j^{-1}$, where the columns of $S_j$ are
\begin{displaymath}
  S_j = (s_j, s_{n+j}).
\end{displaymath}
We obtain a theorem similar to Theorem~\ref{tm3.1}.

\begin{theorem}
  Let $S = (s_1, s_{n+1}, s_2, s_{n+2}, \ldots, s_n, s_{2n}) \in \mathbb{R}^{2m, 2n}$
  with $S_{}^T J S = \widehat{J}(1:2n,1:2n)$ be given.  For $j = 1, \ldots, n$
  let $S_j = (s_j, s_{n+j})$ and $D_j$ as in \eqref{4.2}.  The Frobenius
  norm $\| S_j^{} D_j^{-1} \|_F^2$, $j = 1, \ldots, n$ is minimized for
  \begin{displaymath}
    \widebreve{D}_j^{-1} = \begin{pmatrix}
      \breve{c}_{jj}^{-1} & -\breve{f}_{jj}^{} \\
      0 & \hpm \breve{c}_{jj}^{}
    \end{pmatrix},
  \end{displaymath}
  where
  \begin{displaymath}
    \breve{c}_{jj} = \frac{\| s_j \|_2}{\sqrt[4]{\det(S_j^T S_j^{})}}, \qquad
    \breve{f}_{jj} = \frac{s_j^T s_{n+j}^{}}
          {\| s_j \|_2 \sqrt[4]{\det(S_j^T S_j^{})}}.
  \end{displaymath}
  Thus, for the Frobenius norm of the $j$th block column $S_j$ of $S$
  for the optimal $\widebreve{D}_{j}$ it holds
  \begin{displaymath}
    \|S_j^{} \widebreve{D}_j^{-1} \|_F^{} = \sqrt{2} \delta_j^{}
  \end{displaymath}
  with
  \begin{displaymath}
    \delta_j^{} \assgn \sqrt[4]{\det(S_j^T S_j^{})}.
  \end{displaymath}
  \label{tm4.2}
\end{theorem}
The proof is analogous to the one of Theorem~\ref{tm3.1} and it is
therefore omitted here.  In addition, it is easy to prove that the two
columns of $S_j^{} \widebreve{D}_j^{-1}$ have the same norm.
\begin{corollary}
  It holds that
  \begin{displaymath}
    \| S_j^{} \widebreve{D}_j^{-1} e_1^{} \|_2^{}
    = \| S_j^{} \widebreve{D}_j^{-1} e_2^{} \|_2^{} = \delta_j^{}.
  \end{displaymath}
  \label{cor4.3}
\end{corollary}
\begin{proof}
  The assertion follows immediately,
  \begin{displaymath}
    \| S_j^{} \widebreve{D}_j^{-1} e_1^{} \|_2^{}
    = c_{jj}^{-1} \| s_j^{} \|_2^{}
    = \sqrt[4]{\det(S_j^T S_j^{})}
    = \delta_j
  \end{displaymath}
  and
  \begin{displaymath}
    2 \delta_j^2 = \|S_j^{} \widebreve{D}_j^{-1} \|_F^2
    = \| S_j^{} \widebreve{D}_j^{-1} e_1^{} \|_2^2
    + \| S_j^{} \widebreve{D}_j^{-1} e_2^{} \|_2^2.
    \qedhere
  \end{displaymath}
\end{proof}

Next we state a theorem similar to Theorem~\ref{tm3.6}.
That is, we determine a scaling
\begin{displaymath}
  \widecheck{D}_c = \diag (\widecheck{D}_1, \ldots, \widecheck{D}_n)
    \in \mathcal{D}
\end{displaymath}
such that all columns of the matrix $S\widecheck{D}_c^{-1}$ have the same
Frobenius norm $\delta$.

\begin{theorem}
  Let $S = (s_1, s_{n+1}, s_2, s_{n+2}, \ldots, s_n, s_{2n}) \in \mathbb{R}^{2m, 2n}$
  with $S_{}^T J S = \widehat{J}(1:2n,1:2n)$ be given.  Let $\delta_j$ be as in
  Theorem~\ref{tm4.2}.  Let $\delta \geq \delta_j$.  All columns of
  $S \widecheck{D}_c^{-1}$ have the same norm $\delta$ for
  $\widecheck{D}_c = \diag(\widecheck{D}_1, \ldots, \widecheck{D}_n)
  \in \mathcal{D}$
  where
  \begin{equation}
    \widecheck{D}_j = \begin{pmatrix}
      \check{c}_{jj} & \check{f}_{jj} \\
      0 & \check{c}_{jj}^{-1}
    \end{pmatrix}
    \label{4.3}
  \end{equation}
  for $j = 1, \ldots, n$ with
  \begin{displaymath}
    \check{c}_{jj} = \frac{\|s_j \|_2}{\delta}, \qquad
    \check{f}_{jj} = \frac{s_j^T s_{n+j}^{}
      \pm \sqrt{ \delta_{}^4 - \det(S_j^T S_j^{})}}{\| s_j \|_2 \delta}.
  \end{displaymath}
  \label{tm4.4}
\end{theorem}
The proof is analogous to the one of Theorem~\ref{tm3.6} and it is
therefore omitted here.

Finally, we state the main theorem on the block scaling of $S$ similar
to Theorem~\ref{tm3.8}.
\begin{theorem}
  Let $S = (s_1, s_{n+1}, s_2, s_{n+2}, \ldots, s_n, s_{2n}) \in \mathbb{R}^{2m, 2n}$
  with $S_{}^T J S = \widehat{J}(1:2n,1:2n)$ be given.  Let $\delta_j$ be as in
  Theorem~\ref{tm4.2}.  Let $\delta$ and $\mu$ be defined as
  \begin{displaymath}
    \delta \assgn \max_{j = 1,\ldots, n} \{ \delta_j \}, \qquad
    \mu \assgn \min_{j = 1,\ldots, n} \{ \delta_j \}.
  \end{displaymath}
  Let $\widecheck{D}_c$ and $\widecheck{D}_{j}, j = 1, \ldots, n$
  be as in \eqref{4.3} and Theorem \ref{tm4.4}.  Then $S\widecheck{D}_c^{-1}$
  is nearly optimally scaled.  More precisely, it holds
  \begin{displaymath}
    \min_{D \in \mathcal{D}} \kappa_2^{} (S D_{}^{-1})
    \leq \kappa_2^{}(S \widecheck{D}_c^{-1})
    \leq \sqrt{2n} \,
      \frac{\delta \sqrt{\delta^2 + \sqrt{\delta^4 - \mu^4}}}{\mu^2}
    \min_{D \in \mathcal{D}} \kappa_2^{}(S D_{}^{-1}).
  \end{displaymath}
  \label{tm4.5}
\end{theorem}
The proof is analogous to the one of Theorem \ref{tm3.8} and it is
therefore omitted here.

\begin{remark}
  The optimal choice $\widetilde{D}_c$ from Theorem~\ref{tm4.5} is in
  general not optimal for \eqref{3.2}, that is
  $\kappa_2^{}(\widecheck{D}_r R)$ is not always less or equal to
  $\alpha_R \min_{D \in \mathcal{D}} (D R)$.  See Example \ref{ex6.3}
  for an illustration.
  \label{rm4.6}
\end{remark}

%
%
\section{Connections to related factorizations}
\label{sec:5}
%
%

In the next two subsections we show that the stated results are valid
for the both factors obtained from the symplectic QR factorization of
matrix, and the factor $R$ obtained by the skew-symmetric
(Cholesky-like) factorization of a (skew symmetric) matrix $A$.

%
%
\subsection{Symplectic QR factorization}
\label{subsec:5.1}
%
%

The symplectic QR factorization of a matrix $G \in \mathbb{R}^{2m,2n}$
into the product $QR$ with an upper triangular matrix
$R \in \mathbb{R}^{2n, 2n}$ and an matrix $Q \in \mathbb{R}^{2m, 2n}$
which satisfies $Q_{}^T \widehat{J} Q = \widehat{J}(1:2n,1:2n)$ has
been proposed in~\cite{SingerSanja-SingerSasa-2003}.  If $G_{}^T J G$ is
nonsingular, then $G$ can be factorized as $G \overline{P} = QR$ where
$\overline{P}$ is a suitable permutation matrix.

The result of Section~\ref{sec:3} is valid as stated since the
symplectic QR factorization computes the upper triangular factor $R$.
The results of Section~\ref{sec:4} can be applied to matrix $Q$ since
$P Q = S$.  Therefore, we have
\begin{displaymath}
  S_{}^T J S = (Q P)^T J P Q
  = Q_{}^T \widehat{J} Q = \widehat{J}(1:2n,1:2n)
\end{displaymath}
and, due to unitary equivalence of the spectral norm
\begin{displaymath}
  \kappa_2^{} (Q D_c^{-1}) = \kappa_2^{}(S D_c^{-1})
\end{displaymath}
for $D_c \in \mathcal{D}$.

%
%
\subsection{Skew-symmetric Cholesky-like factorization}
\label{subsec:5.2}
%
%

For any $G \in \mathbb{R}^{2m, 2n}$, $m \geq n$ the matrix $G_{}^T J G$
is skew-symmetric as $J^T = -J$.  Assume that we are given a
permuted SR decomposition of $G$, $G = SR$ with the permuted
symplectic matrix $S$ (that is, $S_{}^T J S = \widehat{J}(1:2n,1:2n)$)
and an upper triangular matrix $R$.  Then
\begin{equation}
  A \assgn G_{}^T J G = R_{}^T S_{}^T J S R = R_{}^T \widehat{J}(1:2n,1:2n) R.
  \label{5.1}
\end{equation}
This factorization of $A$ (almost) corresponds to the Cholesky-like
factorization of skew-symmetric matrices given in~\cite{Bunch-82} (see
also~\cite{Benner-Byers-Fassbender-Mehrmann-Watkins-2000}).  In these
papers it is proven that any skew-symmetric matrix $B \in \mathbb{R}^{2m, 2m}$
whose leading principal submatrices of even dimension are nonsingular
has a unique factorization
\begin{displaymath}
  A = L_{}^T \widehat{J} L
\end{displaymath}
where $L$ is upper triangular with $\ell_{2j-1, 2j} = 0$,
$\ell_{2j-1, 2j-1} > 0$ and $\ell_{2j,2j} = \pm \ell_{2j-1, 2j-1}$ for
$j = 1, \ldots, m$.  Thus $L$ has $2 \times 2$ blocks of the form
\begin{displaymath}
  \begin{pmatrix}
    \ell & 0 \\
    0 & \pm \ell
  \end{pmatrix}
\end{displaymath}
running down the main diagonal.

Thus, if $R$ in \eqref{5.1} is such that its $2 \times 2$ diagonal
blocks are matrices of the form
\begin{displaymath}
  \begin{pmatrix}
    r & 0 \\
    0 & \pm r
  \end{pmatrix}
\end{displaymath}
the decomposition \eqref{5.1} (and hence the SR decomposition of $G$)
is unique (the fact concerning the unique SR decomposition has already
been noted in \cite{Mehrmann-79}).  Moreover, Theorem~\ref{tm3.8} can
be applied to $R$ and we obtain not only an optimal scaled $R$ in the
SR decomposition of $G$, but also the unique Cholesky-like
factorization with optimal block scaling.

But usually, $R$ will have diagonal blocks $R_{jj}$, $j = 1, \ldots, n$
which are upper triangular,
\begin{displaymath}
  R_{jj} = \begin{pmatrix}
    r_{11}^{(j)} & r_{12}^{(j)} \\
    0 & r_{22}^{(j)}
  \end{pmatrix}, \qquad r_{11}^{(j)} r_{22}^{(j)} \neq 0
\end{displaymath}
for $j = 1, \ldots, n$.  Again, Theorem~\ref{tm3.8} can be applied to
$R$ and we obtain not only an optimal scaled $R$ in the SR
decomposition of $G$, but also a non-unique Cholesky-like
factorization with optimal block scaling.

From the factorization $A = L_{}^T \widehat{J} L$ it can be seen that
any scaling matrix $D_L$ applied to $L$ needs to satisfy
\begin{displaymath}
  D_L^T \widehat{J} D_L^{} = \widehat{J}
\end{displaymath}
so that
\begin{displaymath}
  A = L^T \widehat{J} L = (D_L L)^T \widehat{J} (D_L L)
\end{displaymath}
holds.

%
%
\section{Numerical examples}
\label{sec:6}
%
%

In this section we show behavior of the nearly optimal scalings of the
factors $R$ and $S$. The first example shows that the condition number
of the scaled matrix $\widetilde{D}_r R$ can be significantly smaller than
the condition number of $R$, while the second example shows that the bound
\begin{displaymath}
  \alpha_R = \sqrt{2n} \, \frac{\beta \sqrt{\beta^2
      + \sqrt{\beta^4 - \gamma^4}}}{\gamma^2}
\end{displaymath}
can be significantly larger that $1$, and the condition number of the
scaled matrix can rise.
\begin{example}
Let
\begin{displaymath}
  R = \begin{pmatrix}
    a & 0 & a^{-2} & a^{-2} & a^{-2} & a^{-2} \\
      & a & a^{-2} & a^{-2} & a^{-2} & a^{-2} \\
      &   &   a^2 &     0 & a^{-2} & a^{-2} \\
      &   &       &   a^2 & a^{-2} & a^{-2} \\
      &   &       &       & a^{-1} &     0 \\
      &   &       &       &       & a^{-1}
  \end{pmatrix},
\end{displaymath}
be obtained by the SR decomposition, where $a$ is a small parameter,
$0 < a < 1$.

If, for example, $a = 0.1$ then the optimal block-diagonal scaling
from Theorem~\ref{tm3.8} applied from the left to the rows of $R$ is
\begin{displaymath}
  \widetilde{D}_r \approx \begin{pmatrix}
      20.0000 &        -19.9520 &         &                 &        & \\
              & \hpm\hpz 0.0500 &         &                 &        & \\
              &                 & 14.1421 &        -14.0714 &        & \\
              &                 &         & \hpm\hpz 0.0707 &        & \\
              &                 &         &                 & 1.0000 & 0.0000 \\
              &                 &         &                 &        & 1.0000
    \end{pmatrix},
\end{displaymath}
while the final scaled matrix $\widetilde{D}_r R$ is
\begin{align*}
  \widetilde{D}_r R & \approx \begin{pmatrix}
    2.0000 &     -1.9952 & 0.4976 & \hpm 0.4976 & \hpz 0.4976 & \hpz 0.4976 \\
           & \hpm 0.0050 & 5.0000 & \hpm 5.0000 & \hpz 5.0000 & \hpz 5.0000 \\
           &             & 0.1414 &     -0.1407 & \hpz 7.0697 & \hpz 7.0697 \\
           &             &        & \hpm 0.0007 & \hpz 7.0711 & \hpz 7.0711 \\
           &             &        &             &     10.0000 & \hpz 0.0000 \\
           &             &        &             &             &     10.0000
  \end{pmatrix},
\end{align*}
with all row norms equal to $\beta = 10$.  Note that
$\beta_1 \approx 5.3183$, $\gamma = \beta_2 \approx 1.4142$, while
$\beta = \beta_3 = 10$. Therefore, the parameter $\alpha_R$ in the
statement of Theorem~\ref{tm3.8} is $\alpha_R \approx 244.9367$.

For different parameters $a$ we have different values for the
condition numbers of the matrices $R$ and $\widetilde{D}_r R$.
\begin{center}
    \begin{tabular}{@{}ccccc@{}}
      \toprule
      $a$ & $5.0e{-}01$ & $1.0e{-}01$ & $5.0e{-}02$ & $1.0e{-}02$ \\
      \midrule
      $\kappa_2(R)$
        & $5.1810e{+}03$
        & $1.6803e{+}09$
        & $4.1985e{+}11$
        & $1.6080e{+}17$ \\
      $\kappa_2(\widetilde{D}_r R)$
        & $1.5089e{+}03$
        & $1.5829e{+}08$
        & $1.9053e{+}10$
        & $1.3925e{+}15$ \\
      $\beta$
        & $2.3796e{+}00$
        & $1.0000e{+}01$
        & $2.0000e{+}01$
        & $1.0000e{+}02$ \\
      $\gamma$
        & $1.4146e{+}00$
        & $1.4142e{+}00$
        & $1.4142e{+}00$
        & $1.4142e{+}00$ \\
      $\alpha_R$
        & $1.3638e{+}01$
        & $2.4494e{+}02$
        & $9.7978e{+}02$
        & $2.4495e{+}04$ \\
      \bottomrule
    \end{tabular}
\end{center}
\looseness=-1
Since the factor $R$ has quite wildly scaled rows, with the nontrivial
elements in each $2 \times 2$ diagonal block significantly smaller
than the elements in the rest of the corresponding rows, the scaled
triangular factor $\widetilde{D}_r R$ has a significantly lower
condition number than $R$.
\label{ex6.1}
\end{example}

\begin{example}
Let
\begin{displaymath}
  R = \begin{pmatrix}
    a^{-1} &     0 & a^{-1} & a^{-1} & a^{-1} & a^{-1} \\
          & a^{-1} & a^{-1} & a^{-1} & a^{-1} & a^{-1} \\
          &       &     a &     0 &     a &    a \\
          &       &       &     a &     a &    a \\
          &       &       &       & a^{-1} &    0 \\
          &       &       &       &       & a^{-1}
  \end{pmatrix},
\end{displaymath}
be obtained by the SR decomposition, where $a$ is a small parameter,
$0 < a < 1$.

If, for example, $a = 1 \cdot 10^{-1}$ then the optimal block-scaling from
Theorem~\ref{tm3.8} is
\begin{displaymath}
  \widetilde{D}_r \approx \begin{pmatrix}
      1.2910 &     -1.0328 &        &              &        & \\
             & \hpm 0.7746 &        &              &        & \\
             &             & 0.0100 & \hpz 99.9933 &        & \\
             &             &        &     100.0000 &        & \\
             &             &        &              & 0.5774 & 1.6330 \\
             &             &        &              &        & 1.7321
  \end{pmatrix},
\end{displaymath}
while the optimally scaled matrix $\widetilde{D}_r R$ is equal to
\begin{align*}
  \widetilde{D}_r R & \approx \begin{pmatrix}
      12.9099 &        -10.3280 & 2.5820 & \hpz 2.5820 & \hpz 2.5820 & \hpz 2.5820 \\
              & \hpm\hpz 7.7460 & 7.7460 & \hpz 7.7460 & \hpz 7.7460 & \hpz 7.7460 \\
              &                 & 0.0010 & \hpz 9.9993 &     10.0003 &     10.0003 \\
              &                 &        &     10.0000 &     10.0000 &     10.0000 \\
              &                 &        &             & \hpz 5.7735 &     16.3299 \\
              &                 &        &             &             &     17.3205
  \end{pmatrix},
\end{align*}
with all rows-norms equal to $\beta \approx 17.3205$.

For different parameters $a$ we have different values for the
condition numbers of the matrices $R$ and $\widetilde{D}_r R$.
\begin{center}
    \begin{tabular}{@{}ccccc@{}}
      \toprule
      $a$ & $5.0e{-}01$ & $1.0e{-}01$ & $5.0e{-}02$ & $1.0e{-}02$ \\
      \midrule
      $\kappa_2(R)$
        & $5.5000e{+}01$
        & $1.0150e{+}03$
        & $4.0150e{+}03$
        & $1.0002e{+}05$ \\
      $\kappa_2(\widetilde{D}_r R)$
        & $1.3521e{+}02$
        & $7.7471e{+}04$
        & $1.2394e{+}06$
        & $7.7460e{+}08$ \\
      $\beta$
        & $3.4641e{+}00$
        & $1.7321e{+}01$
        & $3.4641e{+}01$
        & $1.7321e{+}02$ \\
      $\gamma$
        & $7.4767e{-}01$
        & $1.4953e{-}01$
        & $7.4768e{-}02$
        & $1.4953e{-}02$ \\
      $\alpha_R$
        & $1.0513e{+}02$
        & $6.5727e{+}04$
        & $1.0516e{+}06$
        & $6.5727e{+}08$ \\
      \bottomrule
    \end{tabular}
\end{center}
This example shows that the optimal scaling, such that all rows have
the same norm, can worsen the condition number of $R$.
\label{ex6.2}
\end{example}

The third example shows that the condition number of
$S\widecheck{D}_r^{-1}$ can be significantly smaller than the condition
number of $S$, while the fourth example shows that the bound
\begin{displaymath}
  \alpha_C = \sqrt{2n} \, \frac{\delta \sqrt{\delta^2
      + \sqrt{\delta^4 - \mu^4}}}{\mu^2}
\end{displaymath}
can be larger than $1$, and the condition number of the scaled matrix
can rise.

Matrices $S$ in the next two examples are computed in the 80-bit
extended precision arithmetic.  The easiest way to produce the
examples is to compute the matrix $Q$ by the symplectic QR factorization
(see~\cite{SingerSanja-SingerSasa-2003}) and then permute the rows,
$S = P Q$, to obtain $S$.  Note that the matrices $R$ are not needed for
conclusion about the optimal scaling of the factor $S$ in the SR
decomposition.  If $G$ is needed, any triangular matrix $R$ will do.
Then $G$ is computed in multiple precision arithemtic as $G = SR$.
\begin{example}
Now suppose that $S$ is computed by the SR decomposition of the matrix
{\scriptsize\begin{displaymath}
  G \approx \begin{pmatrix}
    -8.0000e{-}08 & \hpm 5.9999e{-}10 &     -9.9993e{-}06 &     -2.0816e{-}07 &     -1.0025e{-}05 &     -1.0002e{-}01 \\
\hpm 2.0002e{+}03 &     -9.8412e{+}03 & \hpm 2.1081e{-}01 & \hpm 8.6657e{-}03 & \hpm 1.6001e{+}02 & \hpm 1.0001e{+}03 \\
\hpm 1.9999e{+}00 &     -9.8397e{+}00 &     -1.1008e{+}01 &     -2.2904e{-}01 & \hpm 1.4898e{-}01 &     -1.0097e{-}01 \\
    -1.0000e{-}03 & \hpm 2.0000e{-}05 & \hpm 9.9001e{-}06 & \hpm 1.0208e{-}05 & \hpm 1.0008e{+}00 &     -1.0108e{-}03 \\
\hpm 9.9990e{-}02 & \hpm 7.9902e{-}03 &     -9.9999e{-}01 &     -2.0898e{-}02 & \hpm 6.9991e{-}03 &     -1.0001e{-}01 \\
    -1.9785e{-}02 & \hpm 9.7344e{-}02 & \hpm 1.0003e{+}03 & \hpm 2.0903e{+}01 & \hpm 9.9879e{-}01 & \hpm 1.0003e{+}02
  \end{pmatrix},
\end{displaymath}}
as
{\scriptsize\begin{displaymath}
  S \approx \begin{pmatrix}
    -8.0000e{-}10 & \hpm 7.0000e{-}10 & \hpm 9.9993e{-}06 & \hpm 8.0000e{-}10 & \hpm 9.9999e{-}06 & \hpm 1.0000e{+}00 \\
\hpm 2.0002e{+}01 &     -1.0001e{+}03 &     -2.0900e{-}02 &     -2.0901e{-}09 & \hpm 8.8412e{-}07 & \hpm 9.8014e{-}03 \\
\hpm 1.9999e{-}02 &     -9.9997e{-}01 & \hpm 1.1008e{+}01 & \hpm 1.0010e{-}03 & \hpm 9.8545e{-}10 &     -1.0029e{-}04 \\
    -1.0000e{-}05 & \hpm 1.0000e{-}05 &     -1.0000e{-}05 & \hpm 1.0000e{-}05 &     -1.0000e{-}00 & \hpm 1.0000e{-}05 \\
\hpm 9.9990e{-}04 &     -9.0000e{-}07 & \hpm 1.0000e{+}00 & \hpm 1.0000e{-}07 &     -8.9980e{-}10 &     -1.0010e{-}05 \\
    -1.9785e{-}04 & \hpm 9.8927e{-}03 &     -1.0003e{+}03 &     -1.0992e{-}04 & \hpm 8.9912e{-}07 & \hpm 1.0002e{-}02
  \end{pmatrix}.
\end{displaymath}}
The corresponding $R$ is well-conditioned
{\scriptsize\begin{displaymath}
  R \approx \begin{pmatrix}
    1.0000e{+}02 & 8.0000e{+}00 & \hpm 1.0000e{-}02 &     -7.8600e{-}05 & \hpm 8.0000e{+}00 & \hpm 1.0201e{-}05 \\
                 & 1.0000e{+}01 & \hpm 1.0110e{-}05 &     -9.8000e{-}06 & \hpm 1.0000e{-}05 &     -1.0000e{+}00 \\
                 &              &     -1.0000e{+}00 &     -2.0898e{-}02 &     -1.0001e{-}03 &     -1.0001e{-}01 \\
                 &              &                   & \hpm 9.9988e{-}01 & \hpm 9.0000e{-}06 & \hpm 9.9999e{-}05 \\
                 &              &                   &                   &     -1.0009e{+}00 & \hpm 1.0008e{-}03 \\
                 &              &                   &                   &                   &     -1.0002e{-}01
  \end{pmatrix}.
\end{displaymath}}

The optimal scaling by Theorem~\ref{tm4.5} is obtained by a matrix
$\widecheck{D}_c$, where
\noindent
{\scriptsize\begin{displaymath}
  \widecheck{D}_c \approx \begin{pmatrix}
2.0001e{+}01 &     -1.0001e{+}03 &              &              &              & \\
             & \hpm 4.9997e{-}02 &              &              &              & \\
             &                   & 1.0003e{+}03 & 1.3067e{-}04 &              & \\
             &                   &              & 9.9973e{-}04 &              & \\
             &                   &              &              & 9.9995e{-}01 & 1.7558e{-}08 \\
             &                   &              &              &              & 1.0000e{+}00
  \end{pmatrix}.
\end{displaymath}}

After the optimal scaling we get
{\scriptsize\begin{displaymath}
  S \widecheck{D}_c^{-1} \approx \begin{pmatrix}
    -3.9997e{-}11 &     -7.8606e{-}07 & \hpm 9.9966e{-}09 & \hpm 7.9891e{-}07 & \hpm 1.0000e{-}05 & \hpm 9.9995e{-}01 \\
\hpm 1.0000e{+}00 &     -1.0003e{-}02 &     -2.0895e{-}05 & \hpm 6.4034e{-}07 & \hpm 8.8417e{-}07 & \hpm 9.8009e{-}03 \\
\hpm 9.9990e{-}04 & \hpm 1.0895e{-}05 & \hpm 1.1005e{-}02 & \hpm 9.9978e{-}01 & \hpm 9.8550e{-}10 &     -1.0029e{-}04 \\
    -4.9997e{-}07 &     -9.8007e{-}03 &     -9.9973e{-}09 & \hpm 1.0003e{-}02 &     -1.0000e{+}00 & \hpm 1.0017e{-}05 \\
\hpm 4.9992e{-}05 & \hpm 9.9995e{-}01 & \hpm 9.9973e{-}04 &     -3.0639e{-}05 &     -8.9984e{-}10 &     -1.0009e{-}05 \\
    -9.8920e{-}06 &     -1.0779e{-}07 &     -9.9999e{-}01 & \hpm 2.0753e{-}02 & \hpm 8.9917e{-}07 & \hpm 1.0002e{-}02
  \end{pmatrix}.
\end{displaymath}}
In this case
\begin{displaymath}
  \kappa_2(S) = 1.0327e{+}06, \qquad
  \kappa_2(S\widecheck{D}_c^{-1}) = 1.0623, \qquad
  \delta = 1.000049, \quad
  \mu = 1.000024, \quad
\end{displaymath}
and the row-norms are equal to $1.000049$ while $\alpha_C = 3.4815$.
Note that in this case we have a very precise estimation of the
maximal condition number over all block diagonal scalings of
the form \eqref{3.1}.

If the matrix $S$ is scaled by the factor $\widetilde{D}_r^{-1}$ from
Example~\ref{ex6.1}, instead of $\widecheck{D}_c^{-1}$, then
$\kappa_2(S\widetilde{D}_r^{-1}) \approx 3.8465e{+10}$.  In the case
of $\widetilde{D}_r^{-1}$ from Example~\ref{ex6.2} the condition
number is even higher,
$\kappa_2(S\widetilde{D}_r^{-1}) \approx 2.0251e{+14}$.

On the other hand, if $\widecheck{D}_c$ is used to scale $R$ from
Example~\ref{ex6.1} we get
$\kappa_2(\widecheck{D}_c R) \approx 5.4894e{+20}$.  For $R$ from
Example~\ref{ex6.2} the result is very similar,
$\kappa_2(\widecheck{D}_c R) \approx 2.29358e{+17}$.
\label{ex6.3}
\end{example}

\begin{example}
Now suppose that $S$ is computed by the SR decomposition of $G$,
{\scriptsize\begin{displaymath}
  G \approx \begin{pmatrix}
\hpm 1.0871e{+}02 & \hpm 1.4643e{+}01 &     -5.4969e{-}01 &     -1.1806e{-}02 & \hpm 9.2375e{+}00 &     -6.5123e{-}01 \\
    -5.2820e{+}01 &     -8.8338e{+}00 & \hpm 5.8813e{-}01 & \hpm 1.3947e{+}00 &     -2.8501e{+}00 & \hpm 4.1022e{-}01 \\
    -1.8322e{+}01 & \hpm 1.5381e{+}00 &     -5.1659e{-}02 &     -9.0207e{-}01 &     -1.8338e{+}00 &     -2.9221e{-}01 \\
    -5.9464e{+}01 & \hpm 1.1893e{+}00 &     -5.9404e{-}03 & \hpm 4.0911e{-}05 &     -4.2155e{+}00 &     -5.9519e{-}01 \\
\hpm 3.7614e{+}01 & \hpm 3.0091e{+}00 &     -3.9718e{-}01 &     -8.4084e{-}03 & \hpm 3.7575e{+}00 & \hpm 4.9976e{-}04 \\
\hpm 6.1056e{+}01 & \hpm 4.3350e{+}00 &     -1.7096e{+}00 & \hpm 1.2893e{-}01 & \hpm 6.0988e{+}00 &     -5.6762e{-}02
  \end{pmatrix},
\end{displaymath}}
as
\begin{displaymath}
  S \approx \begin{pmatrix}
\hpm 1.0871 & \hpm 0.5946 & \hpm 0.5606 & \hpm 0.0000 &     -0.5411 & -1.08e{-}19 \\
    -0.5282 &     -0.4608 &     -0.5934 & \hpm 1.3825 &     -1.3738 & \hpm 1.0868 \\
    -0.1832 & \hpm 0.3004 & \hpm 0.0498 &     -0.9011 & \hpm 0.3677 &     -0.1288 \\
    -0.5946 & \hpm 0.5946 & \hpm 0.0000 & \hpm 0.0000 &     -0.5411 & \hpm 0.0000 \\
\hpm 0.3761 & 1.02e{-}20  & \hpm 0.4009 & -6.78e{-}21 &     -0.7482 &     -0.4133 \\
\hpm 0.6106 &     -0.0550 & \hpm 1.7157 & \hpm 0.1649 &     -1.2150 &     -0.6106
  \end{pmatrix}.
\end{displaymath}
The corresponding $R$ is equal to one from Example~\ref{ex6.3}.

The optimal scaling of rows of $S$ is obtained by a block diagonal
matrix $D_c$,
\begin{displaymath}
  \widecheck{D}_c \approx \begin{pmatrix}
0.8634 & 1.1876 &        &             &        & \\
       & 1.1582 &        &             &        & \\
       &        & 1.0913 &     -0.1685 &        & \\
       &        &        & \hpm 0.9164 &        & \\
       &        &        &             & 1.2107 & 0.2583 \\
       &        &        &             &        & 0.8260
  \end{pmatrix}.
\end{displaymath}
The scaled matrix
\begin{displaymath}
  S \widecheck{D}_c^{-1} \approx \begin{pmatrix}
\hpm 1.2590 &     -0.7775 & \hpm 0.5137 & \hpm 0.0944 &     -0.4470 & \hpm 0.1398 \\
    -0.6117 & \hpm 0.2294 &     -0.5438 & \hpm 1.4087 &     -1.1347 & \hpm 1.6706 \\
    -0.2122 & \hpm 0.4769 & \hpm 0.0457 &     -0.9750 & \hpm 0.3037 &     -0.2509 \\
    -0.6887 & \hpm 1.2196 & \hpm 0.0000 & \hpm 0.0000 &     -0.4470 & \hpm 0.1398 \\
\hpm 0.4356 &     -0.4467 & \hpm 0.3674 & \hpm 0.0675 &     -0.6180 &     -0.3072 \\
\hpm 0.7071 &     -0.7725 & \hpm 1.5722 & \hpm 0.4689 &     -1.0036 &     -0.4254
  \end{pmatrix}
\end{displaymath}
has a somewhat higher condition number than the original $S$.  Indeed,
we have
\begin{displaymath}
  \kappa_2(S) = 18.0149, \qquad
  \kappa_2(S\widecheck{D}_c^{-1}) = 21.9625, \qquad
  \delta = 1.7800, \qquad
  \mu = 1.2168,
\end{displaymath}
with the row-norms equal to $1.7800$, and $\alpha_C = 10.1756$.
\label{ex6.4}
\end{example}

%
%
\section{Concluding remarks}
\label{sec:7}
%
%

The results of this paper may help to refine the relative perturbation
results for the eigendecomposition of skew-symmetric matrices computed
by the algorithm derived by Pietzsch in his PhD
thesis~\cite{Pietzsch-93}.

%
%


\end{document}